\documentclass[10pt]{article} 
\usepackage{amsthm,hyperref}
\usepackage{graphicx,tikz,color}
\usepackage{amssymb}
\usepackage{caption,hyperref}
\usepackage{subcaption}
\usepackage{amssymb, amsmath}
\usepackage{enumerate}
\newcommand{\RR}{\mathbb R}
\newcommand{\NN}{\mathbb N}

\newcommand{\T}{\mathcal T}
\newcommand{\R}{\mathcal R}

\newtheorem{theorem}{Theorem}[section]
\newtheorem{definition}[theorem]{Definition}

\newtheorem{lemma}[theorem]{Lemma}
\newtheorem{remark}[theorem]{Remark}

\newtheorem{example}[theorem]{Example}

\newtheorem{assumption}{Assumption}

\newcommand\blfootnote[1]{%
	\begingroup
	\renewcommand\thefootnote{}\footnotetext{#1}%
	\addtocounter{footnote}{-1}%
	\endgroup
}

\begin{document}

\title{Internal 
	observability of the wave equation in tiled domains}
%

\author{Anna Chiara Lai
\\
              Dipartimento di Scienze di Base
              e Applicate per l'Ingegneria,\\
              Sapienza Universit\`{a} di  Roma\\
              via A. Scarpa, 16,\\
              00161  Roma, Italy          \\ \texttt{anna.lai@sbai.uniroma1.it}            }
%

\maketitle
\blfootnote{This research is supported by Sapienza Universit\`{a} di Roma, Dipartimento di Scienze di Base e Applicate per l'ingegneria, Assegno di Ricerca n. 7/2016.}

\begin{abstract}
 We investigate the internal observability of the wave equation with Dirichlet 
boundary conditions in tilings. The paper includes a general result 
relating internal observability problems in general domains to their tiles, and a discussion of the case in which the domain is the $30$-$60$-$90$ triangle. 
\end{abstract}

\section{Introduction}\label{s1}
The aim of the present paper is to investigate internal observability properties of vibrating repetitive structures.
Motivated by applications of hexagonal and triangular tilings (and related subtilings) to engineering, the particular
case of the half to the equilateral triangle is treated in detail.

By a repetitive structure, or tessellation, is meant a structure obtained by the assemblation of identical substructures, 
or tiles. For instance, two-dimensional lattices and the honeycomb lattice are examples of tessellation of $\RR^2$; 
while the regular hexagon (i.e., the tile of the honeycomb lattice) and the rectangle with aspect ratio equal to $\sqrt{3}$ 
are bounded domains that can be both tiled with $30$-$60$-$90$ triangles, see Figure \ref{fig1}.
The interest in repetitive systems of vibrating membranes is motivated by applications in mechanical, civil and aerospace
engineering \cite{multisymmetric,cooling}. Modular structures have indeed the double advantage of a cost-effective
manufacturing and construction (due to the repetitivity of the process) as well as a computationally cost-effective design. 
In particular, structural eigenproblems (e.g., vibrations and buckling) for repetitive structures in general involve
a lower number of degrees of freedoms and, consequently, a less computationally demanding numerical solution 
\cite{forcedvibration}. 
Tilings involving regular triangles and hexagons (known as triangular lattice and honeycomb lattice, respectively)  
find  countless applications in engineering, as well \cite{elegia}. 
For instance, the use of such structures in architectural engineering  is motivated by their mechanical properties, 
including resistance to external load and energy absorption, see for instance \cite{prop1,prop2} and, for a comprehensive dissertation on the topic,
the book \cite{cellular}.
Finally, we mention that honeycomb lattice plays a crucial role in nanosciences and, in particular, in graphene technology \cite{graphene}. 

\begin{figure}[t!]
	\centering
	\begin{subfigure}[t]{0.48\textwidth}
		\centering
		\includegraphics[scale=0.5]{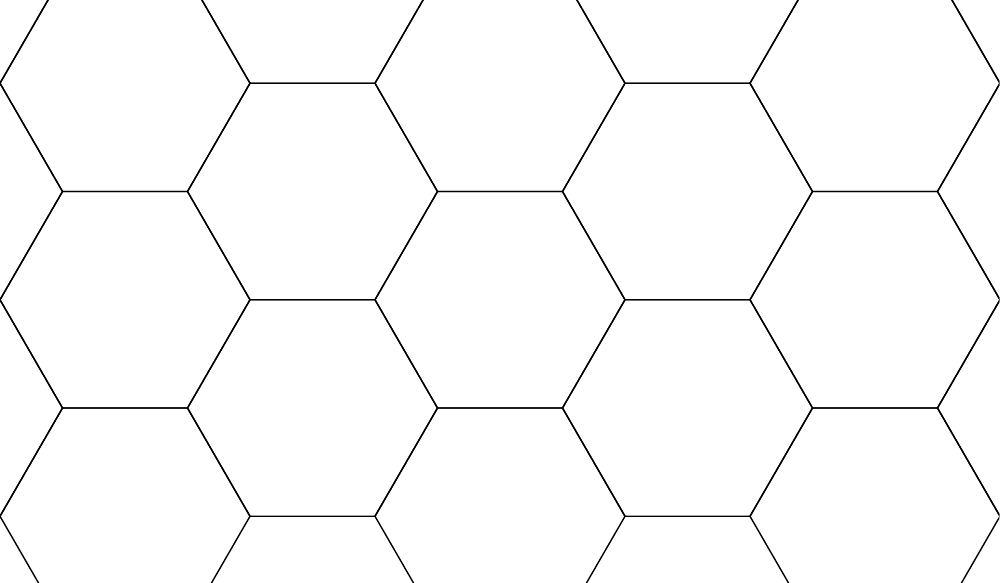}
		\caption{Honeycomb tiling}
	\end{subfigure}%
		\begin{subfigure}[t]{0.48\textwidth}
		\centering
		\includegraphics[scale=0.5]{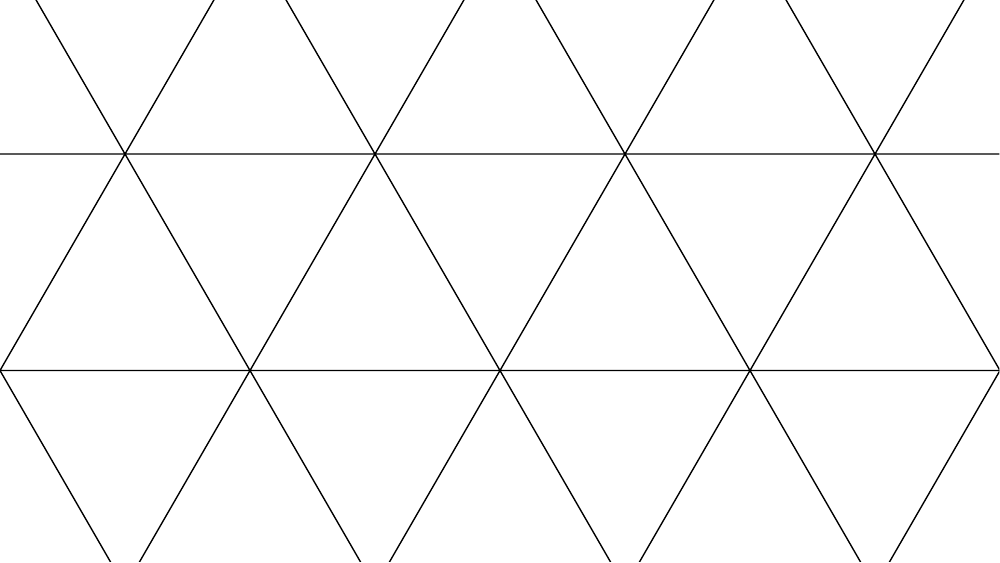}
		\caption{Triangular tiling}
	\end{subfigure}\\
	\hskip-1cm\begin{subfigure}[t]{0.5\textwidth}
\hskip1.55cm
		\includegraphics[scale=0.2]{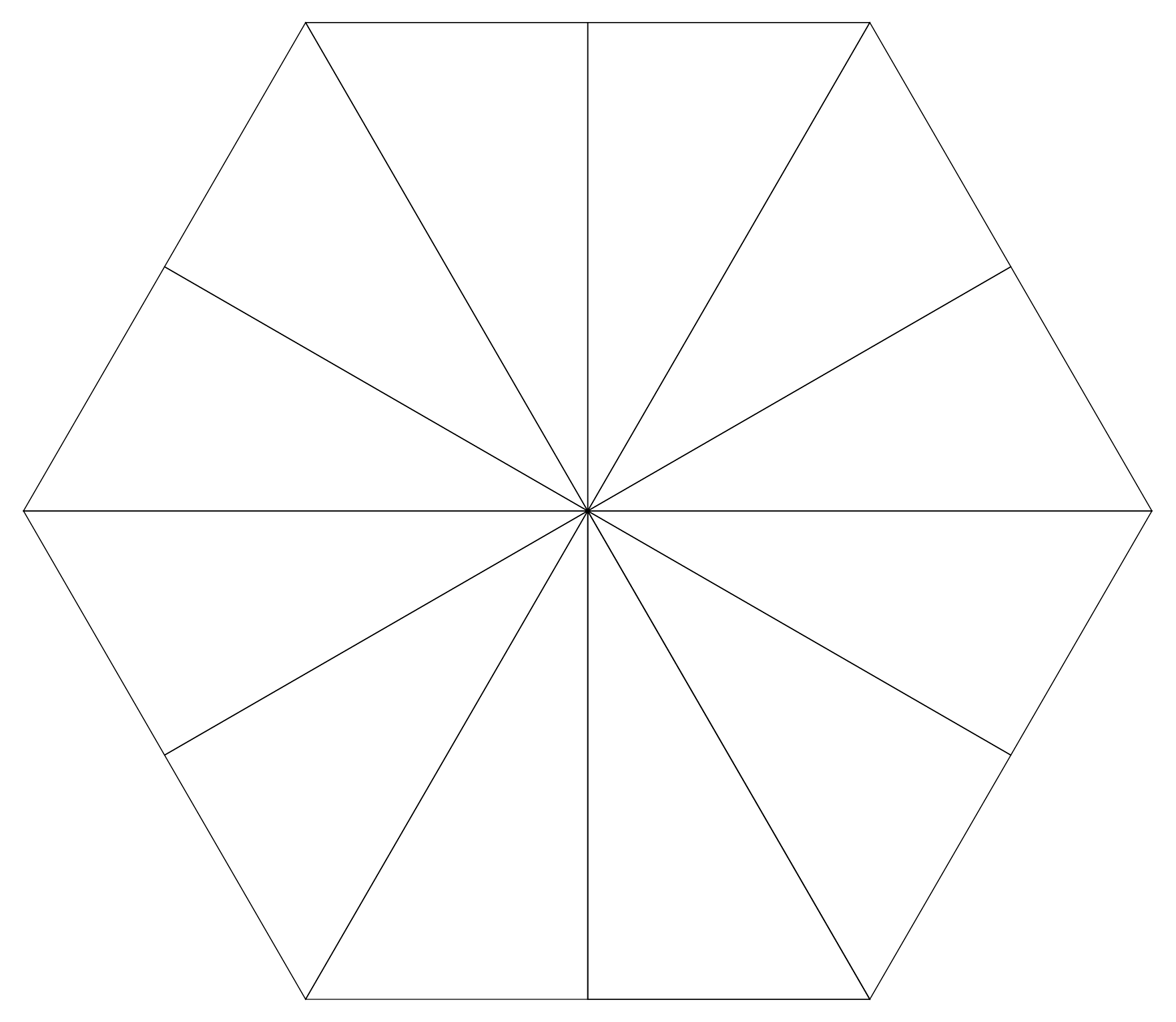}
		\caption{Tiling of the regular hexagon}
	\end{subfigure}%
	\hskip0.5cm
	\begin{subfigure}[t]{0.4\textwidth}
		\centering
		\includegraphics[scale=0.5]{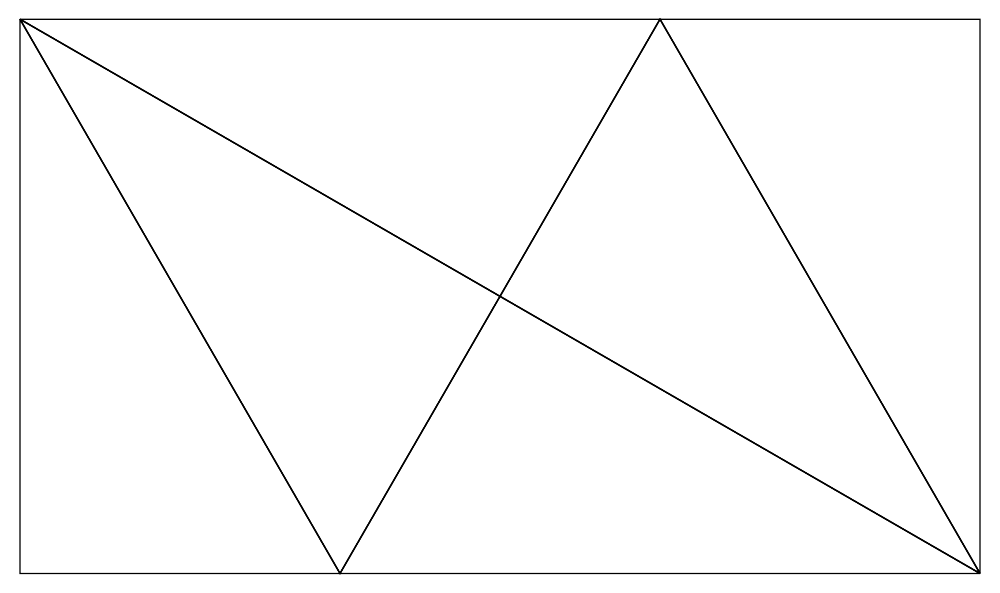}
		\caption{Tiling of the rectangle with aspect ratio equal to $\sqrt{3}$}
	\end{subfigure}
	\caption{Some tilings related to the $30$-$60$-$90$ triangle.\label{fig1}}
	
\end{figure}

As mentioned above, we are interested in the internal observability of the wave equation, 
that is the problem of reconstructing initial data from the observation of the evolution of the system in a subregion of the domain.
Using folding and tessellation techniques, in the spirit of \cite{triangle} and \cite{KS84}, 
we provide a general class of tilings, called \emph{admissible tilings}, for which some internal
observability properties of tiled domains extend to their tiles and -- under some symmetry assumptions on initial data -- vice versa.  
In particular, we show how to bridge the well-established theory
concerning rectangular domains  \cite{Har1989,Har1991,KomLor152,KomLor159,KomMia2013,Meh2009} to the case of a $30$-$60$-$90$ triangular domain. In the remaining part of this Introduction we discuss
in detail this case, while postponing the more technical, general result to Section \ref{sgeneral}. 

\subsection{A case study: observability in a triangular domain}

%

We consider the problem 
\begin{equation}\label{wave}
\begin{cases}
u_{tt}-\Delta u=0& \text{in }\RR\times \T\\
u=0&\text{in }\RR\times \partial \T\\
u(t,0)=u_0,~u_t(t,0)=u_1&\text{in }\T
\end{cases}
\end{equation}
where
$\mathcal T$ is the open triangle with vertices $(0,0),(1/\sqrt{3},0)$ and $(0,1)$.
Also consider ther rectangle $\mathcal  R:=(0,\sqrt{3})\times (0,1)$ and remark that there exists $6$ rigid transformations $K_1,\dots,K_6$ satisfying the relation
$$cl(\mathcal R)=\bigcup_{h=1}^6 K_h (cl(\T)),$$ 
where $cl(\Omega)$ represent the closure of a set $\Omega$  -- see Figure \ref{tiling}. We then say that $\T$ \emph{tiles}
\footnote{ For a precise definition of tilings see Definition \ref{deftiling} below, while  the explicit 
	definition of the $K_h$'s is given in Section \ref{st1}.} $\mathcal R$. 
%
%
%
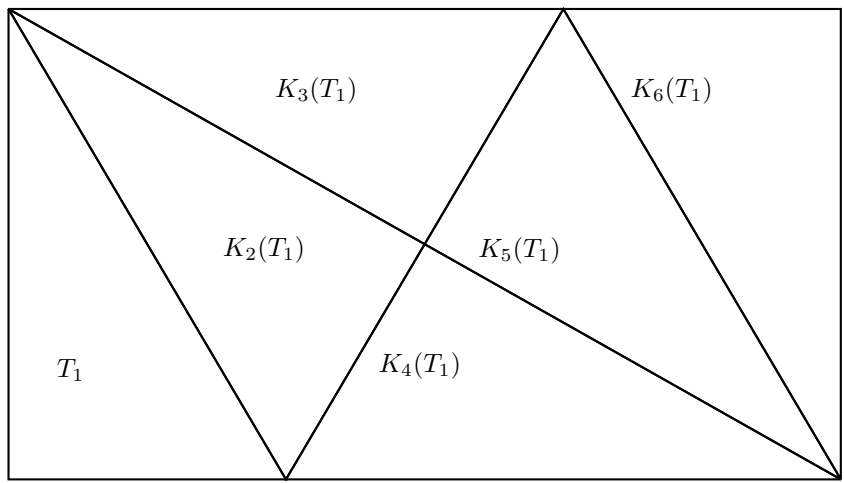
\begin{figure}\centering
	\begin{tikzpicture}[y=0.80pt, x=0.80pt, yscale=-1.000000, xscale=1.000000, inner sep=0pt, outer sep=0pt]
\begin{scope}[cm={{0.64759,0.0,0.0,-0.63445,(-725.19037,529.07247)}}]
  \begin{scope}
    \path[draw=black,line join=miter,line cap=rect,miter limit=3.25,line
      width=0.800pt] (12.6600,363.8480) -- (215.2190,13.0040) -- (12.6600,13.0040)
      -- (12.6600,363.8480) -- cycle;

    \path[draw=black,line join=miter,line cap=rect,miter limit=3.25,line
      width=0.800pt] (12.6600,363.8480) -- (417.7810,363.8480) --
      (316.5000,188.4260) -- (12.6600,363.8480) -- cycle;

    \path[draw=black,line join=miter,line cap=rect,miter limit=3.25,line
      width=0.800pt] (620.3400,13.0040) -- (215.2190,13.0040) -- (316.5000,188.4260)
      -- (620.3400,13.0040) -- cycle;

    \path[draw=black,line join=miter,line cap=rect,miter limit=3.25,line
      width=0.800pt] (620.3400,13.0040) -- (417.7810,363.8480) --
      (620.3400,363.8480) -- (620.3400,13.0040) -- cycle;

    \path[draw=black,line join=miter,line cap=rect,miter limit=3.25,line
      width=0.800pt] (12.6600,363.8480) -- (215.2190,13.0040) -- (316.5000,188.4260)
      -- (12.6600,363.8480) -- cycle;

    \path[draw=black,line join=miter,line cap=rect,miter limit=3.25,line
      width=0.800pt] (620.3400,13.0040) -- (417.7810,363.8480) --
      (316.5000,188.4260) -- (620.3400,13.0040) -- cycle;

  \end{scope}
\end{scope}
\path[xscale=1.010,yscale=0.990,fill=black,line join=miter,line cap=butt,line
  width=0.800pt] (-687.1593,478.9459) node[above right] (text4187) {$T_1$};

\path[xscale=1.010,yscale=0.990,fill=black,line join=miter,line cap=butt,line
  width=0.800pt] (-609.4599,422.4844) node[above right] (text4191) {$K_2(T_1)$};

\path[xscale=1.010,yscale=0.990,fill=black,line join=miter,line cap=butt,line
  width=0.800pt] (-489.8859,422.7302) node[above right] (text4191-3)
  {$K_5(T_1)$};

\path[xscale=1.010,yscale=0.990,fill=black,line join=miter,line cap=butt,line
  width=0.800pt] (-585.0149,346.5223) node[above right] (text4191-7)
  {$K_3(T_1)$};

\path[xscale=1.010,yscale=0.990,fill=black,line join=miter,line cap=butt,line
  width=0.800pt] (-536.5394,478.0724) node[above right] (text4191-35)
  {$K_4(T_1)$};

\path[xscale=1.010,yscale=0.990,fill=black,line join=miter,line cap=butt,line
  width=0.800pt] (-418.6524,346.5535) node[above right] (text4191-9)
  {$K_6(T_1)$};

\end{tikzpicture}
	\caption{The tiling of $\mathcal R$ with $\mathcal T$.  Note that $K_1$ is the identity map, hence $K_1(\mathcal T)=\mathcal T$. \label{tiling}}\end{figure}

As it is well known, a complete orthonormal base for $L^2( \mathcal R)$ is given by the
eigenfunctions of $-\Delta$ in $H_0^1( \mathcal R)$
$$\overline e_{k}:=\sin(\pi k_1 x_1/\sqrt{3})\sin (\pi k_2 x_2), \quad \text{where } k=(k_1,k_2),~k_1,k_2\in \NN$$
and the associated eigenvalues are $\gamma_k=\frac{k_1^2}{3}+k_2^2$. 
In \cite{triangle}, a folding technique (that we recall in detail in Section 
\ref{st1}) is used to derive from $\{\overline e_{k}\}$ an orthogonal base 
$\{e_{k}\}$ of $L^2(\mathcal T)$ formed by the eigenfunctions of $-\Delta$ in 
$H_0^1(\mathcal T)$. 
The explicit knowledge of a eigenspace for $H_0^1(\mathcal T)$ allows us 
to set the problem \eqref{wave} in the framework 
of Fourier analysis -- see \cite{KomLorbook,Ing1936,Har1989,Har1991,Beu,BaiKomLor103,BaiKomLor111}. Our goal is to exploit the deep relation between the 
eigenfunctions for $H_0^1(\mathcal R)$ and those of $H_0^1(\mathcal T)$ in order 
to extend known observability results for $\mathcal R$ to $\mathcal T$.

In particular, we are interested in the \emph{internal observability} of 
\eqref{wave}, i.e., in the validity of the estimates 
$$\|u_0\|_{L^2(\T)}^2+\|u_1\|^2_{H^{-1}(\T)}\asymp 
\int_0^T\int_{\T_0} |u(t,x)|^2 dx $$
where $\T_0$ is a subset of $\T$ and $T$ is sufficiently large.
Here 
and in the sequel $A\asymp B$ means $c_1 A\leq B \leq c_2 A$ with some 
constants 
$c_1$ and $c_2$ which are independent from $A$ and $B$. When we need to stress 
the dependence of these estimates on the couple of constants $c=(c_1,c_2)$, we 
write $A \asymp_{c} B$. Also by writing $A\leq_c B$ we mean
the inequality $ c A\leq B$ while the expression $A\geq_c B$ denotes $c A\geq 
B$.

We have\begin{theorem}\label{thmmainintro}
	Let  $\overline u$ be the solution of
	\begin{equation}\label{wavegeneralrectintro}
	\begin{cases}
	\overline u_{tt}-\Delta \overline u=0& \text{in }\RR\times  {\mathcal R}\\
	\overline u=0&\text{in }\RR\times \partial \mathcal R\\
	\overline u(t,0)= \overline u_0,~u_t(t,0)=\overline u_1&\text{in } {\mathcal R},
	\end{cases}
	\end{equation}
	let $\mathcal R_0$ be a subset of $\mathcal R$ and assume that there exists a 
	constant $T_0\geq 0$ such that if $T>T_0$ then there 
	exists a couple of constants $c=(c_1,c_2)$ such that  $\overline u$ satisfies
	\begin{equation}\label{obsrectanglegintro}
	\|\overline u_0\|_{L^2(\R)}^2+\|\overline u_1\|^2_{ H^{-1}(\R)}\asymp_{c} 
	\int_0^T\int_{\mathcal R_0} |\overline u(t,x)|^2 dx  dt
	\end{equation}
	for all $(\overline u_0, \overline u_1)\in L^2(\mathcal R)\times 
	H^{-1}(\R)$. 
	Moreover
	let
	$$\T_0:=\bigcup_{h=1}^6 K^{-1}_h( \R_0)\cap \mathcal T.$$
	Then for each $T>T_{0}$ and $(u_0,u_1)\in L^2(\mathcal T)\times H^{-1}(\mathcal T)$, the solution $u$ of \eqref{wave}  satisfies
	\begin{equation}\label{obstrgintro}
	\| u_0\|_{L^2(\T)}^2+\| u_1\|^2_{ H^{-1}(\T)}\asymp_{c} 
	\int_0^T\int_{\mathcal T_0} |\overline u(t,x)|^2 dx  dt
	\end{equation}

	The result also holds by replacing every occurrence of $\asymp_c$ with $\leq_c$ 
	or $\geq_c$.
\end{theorem}
We point out that the time of observability $T_{0}$ stated in Theorem 
\ref{thmmainintro}, as well as the couple $c$ of constants in the estimates 
\eqref{obsrectanglegintro} and \eqref{obstrgintro}, are the same for both the 
domains $\mathcal R$ and $\mathcal T$. Also note that in Section \ref{st1} we 
prove a slightly stronger version of Theorem \ref{thmmainintro}, that is 
Theorem \ref{thmmain}: its precise statement requires some technicalities that 
we chose to avoid here, however we may anticipate to the reader that 
the assumption on initial data $(\overline u_0,\overline u_1)\in  L^2(\mathcal 
R)\times H^{-1}(\mathcal R)$ can be weakened by replacing $ L^2(\mathcal 
R)\times H^{-1}(\mathcal R)$ with an appropriate subspace. 


\subsection{Organization of the paper.} In 
Section \ref{sgeneral} we consider a 
generic domain $\Omega$ tiling a larger domain $\Omega'$:  we establish a 
result, Theorem \ref{thmprol}, relating the observability properties of wave 
equation on $\Omega'$ and on its tile $\Omega$. Section \ref{st1} is devoted to the proof of Theorem \ref{thmmainintro}.

\section{An observability result on tilings}\label{sgeneral}
The goal of this section is to state an equivalence between an observability 
problem on a domain $\Omega$ and an observability problem on a larger domain 
$\Omega'$, under the assumption that $\Omega$ tiles $\Omega'$. 
We begin with some definitions.
\begin{definition}[Tiling]\label{deftiling}
	Let $\Omega$ and $\Omega'$ be two open bounded subsets of $\RR^n$. We say 
	that $\Omega$ \emph{tiles} $\Omega'$ if there 
	exists a set $\{K_h\}_{h=1}^N$ rigid transformations of $\RR^n$ 
	such that
	$$cl(\Omega')=\bigcup_{h=1}^N K_h(cl(\Omega))$$
	and such that $K_h(\Omega)\cap K_j(\Omega)=\emptyset$ for all $h\not=j$.
\end{definition}
\begin{definition}[Foldings and prolongations] 
	Let $(\Omega,\{K_h\}_{h=1}^N)$ be a tiling of $\Omega'$ and 
	$\delta=(\delta_1,\dots,\delta_N)\in\{-1,1\}^N$. The \emph{prolongation with 
		coefficients $\delta$} of a function $u:\Omega\to \RR$ to $\Omega'$ is the 
	function $\mathcal P_\delta u:\Omega'\to \RR$ 
	$$\mathcal P_\delta u(K_h x)=\delta_h u(x) \qquad \forall h=1,\dots,N.$$
	The \emph{folding with coefficients $\delta$} of a function $\overline 
	u:\Omega'\to \RR$ is the function $\mathcal F_\delta \overline u:\Omega\to 
	\RR$ defined by
	$$\mathcal F_\delta \overline u(x)=\frac{1}{N^2}\sum_{h=1}^N \delta_h 
	\overline u(K_h x) \qquad \forall h=1,\dots,N.$$
	When the particular choice of $\delta$ is not relevant we omit it in the 
	under scripts and we simply write $\mathcal P$ and $\mathcal F$. 
\end{definition}
\begin{definition}[Admissible tiling]
	A tiling $(\Omega,\{K_h\}_{h=1}^N)$ of $\Omega'$ is \emph{admissible} if there exists 
	$\delta\in\{-1,1\}^N$ such that 
	\begin{equation}\label{admdef}
	\mathcal F_\delta\varphi \in H^1_0(\Omega) \quad \forall \varphi \in 
	H_0^1(\Omega').
	\end{equation}
\end{definition}

\begin{example}\label{exadm}
	We show in Lemma \ref{ladm} below that the tiling of $\mathcal R$ with 
	$\mathcal T$ depicted in Figure \ref{tiling} is admissible, in particular 
	\eqref{admdef} holds with $\delta=(1,-1,1,1,-1,1)$.
	
	On the other hand the tiling of $\mathcal R':=(0,1/\sqrt{3})\times(0,1)$ 
	given 
	by the transformations $K_1':=id$ and $$K_2':(x_1,x_2)\mapsto 
	-(x_1,x_2)+(1/\sqrt{3},1),$$
	see Figure \ref{tilingbad}, is not admissible. Let indeed  
	$v_1:=(1/\sqrt{3},0)$, 
	$v_2:=(0,1)$ and $x_\lambda:=\lambda v_1 +(1-\lambda)v_2$ with $\lambda\in 
	(0,1)$. 
	Then $x_\lambda\in \partial \mathcal T$ and 
	$$K_2(x_{\lambda})=x_{1-\lambda}$$
	Therefore it suffices to choose $\varphi\in H_0^1(\mathcal R)$ such that 
	$\varphi(x_\lambda)\not=\pm \varphi(x_{1-\lambda})$ to obtain 
	$$\mathcal F_\delta \varphi(x_\lambda)=\delta_1 \varphi(x_\lambda)+\delta_2 
	\varphi(x_{1-\lambda})\not=0$$
	for all $\delta_1,\delta_2\in\{-1,1\}$. Consequently $\mathcal F_\delta 
	\varphi \not\in H_0^1(\mathcal T)$ for all $\delta\in\{-1,1\}^2$.
\end{example}

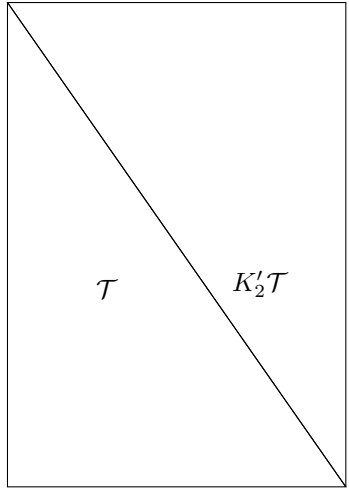
\begin{figure}
\centering	\definecolor{cffffff}{RGB}{255,255,255}

\begin{tikzpicture}[y=0.80pt, x=0.80pt, yscale=-1.000000, xscale=1.000000, inner sep=0pt, outer sep=0pt]
\begin{scope}[cm={{0.65622,0.0,0.0,-0.54271,(349.54106,742.37215)}}]
  \begin{scope}
    \path[fill=cffffff,nonzero rule] (0.0000,0.0000) -- (254.0000,0.0000) --
      (254.0000,432.0000) -- (0.0000,432.0000) -- (0.0000,0.0000) -- cycle;

    \path[draw=cffffff,line join=miter,line cap=butt,miter limit=10.00,line
      width=0.024pt] (0.0000,432.0000) -- (254.0000,432.0000) -- (254.0000,0.0000)
      -- (0.0000,0.0000) -- (0.0000,432.0000) -- cycle;

    \path[draw=black,line join=miter,line cap=rect,miter limit=3.25,line
      width=0.400pt] (5.0780,427.1250) -- (248.8630,4.8750) -- (5.0780,4.8750) --
      (5.0780,427.1250) -- cycle;

    \path[draw=black,line join=miter,line cap=rect,miter limit=3.25,line
      width=0.400pt] (248.8630,4.8750) -- (5.0780,427.1250) -- (248.8630,427.1250)
      -- (248.8630,4.8750) -- cycle;

  \end{scope}
\end{scope}
\path[xscale=1.140,yscale=0.877,fill=black,line join=miter,line cap=butt,line
  width=0.800pt] (346.3738,741.5263) node[above right] (text4200) {$\mathcal
  T$};

\path[xscale=1.140,yscale=0.877,fill=black,line join=miter,line cap=butt,line
  width=0.800pt] (402.7135,741.2077) node[above right] (text4200-3)
  {$K_2'\mathcal T$};

\end{tikzpicture}
	\caption{\label{tilingbad} A non-admissible tiling of $\mathcal 
		R'=(0,1/\sqrt{3})\times(0,1)$ with $\mathcal T$. }
\end{figure}
\begin{remark}
	We borrowed the notion of prolongation and folding from \cite{triangle}: 
	while our definition of $\mathcal P_\delta$ is exactly as it is given in 
	\cite{triangle}, we introduced a normalizing term $1/N^2$ in the definition of 
	$\mathcal F_\delta$ in order to enlighten the notations. Note that the 
	following 
	equality holds:
	\begin{equation}\label{fold2}
	\mathcal F_\delta(\mathcal P_\delta u)=\frac{1}{N}u  
	\end{equation}
	for all $u:\Omega\to\RR$.

	Also remark that we shall need to prolong and fold also functions 
	$u:\RR\times \Omega\to \RR$ and 
	$\bar u:\RR\times \Omega'\to \RR$, in this case the definition of $\mathcal 
	P$ and $\mathcal F$ naturally extends by applying the transformations $K_h$'s 
	to the spatial variables $x$. For instance
	if  $u:\RR\times \Omega\to \RR$ then its prolongation to $\RR\times 
	\Omega'$ reads
	$$\mathcal P_\delta u(t,K_h x)=\delta_h u(t,x).$$ 
\end{remark}
%
%
%

We want to establish a relation between solutions of a wave equation with 
Dirichlet boundary conditions and their prolongation. To this end we introduce 
the notations
$$\mathcal P_\delta L^2(\Omega):=\{\mathcal P_\delta u\mid u\in 
L^2(\Omega)\},$$
$$\mathcal P_\delta H^1_0(\Omega):=\{\mathcal P_\delta u\mid u\in 
H^1_0(\Omega)\}$$
and
$$\mathcal P_\delta H^{-1}(\Omega):=\{\mathcal P_\delta u\mid u\in 
H^{-1}(\Omega)\}.$$

Note that $\mathcal P_\delta L^2(\Omega)\subset L^2(\Omega')$, $\mathcal 
P_\delta H^1_0(\Omega)\subset H^1_0(\Omega')$ and $\mathcal P_\delta 
H^{-1}(\Omega)\subset H^{-1}(\Omega)$ .

All results below hold under the following assumptions on the domains 
$\Omega$, $\Omega'$ and on a base $\{e_k\}$ for $L^2(\Omega)$:
\begin{assumption}\label{A1}
	$(\Omega,\{K_h\}_{h=1}^N)$ is an admissible tiling of $\Omega'$.
\end{assumption}

\begin{assumption}\label{A2}
	$\{e_k\}$ is a base of eigenvectors of $-\Delta$ in $H^1_0(\Omega)$, it is 
	defined on $\Omega\cup \Omega'$ and there exists $\delta\in\{-1,1\}^N$ such 
	that
	$$\mathcal P_\delta (e_k|_{\Omega})=e_k|_{\Omega'}$$ 
	for each $k\in\NN$.
\end{assumption}
\begin{remark}[Some remarks on Assumption \ref{A2}]
	We note that Assumption \ref{A2} can be equivalently stated as
	\begin{equation}\label{opdef}
	e_k(K_h x)=\delta_h e_k(x)\quad \text{
		for all $x\in \Omega$, $h=1,\dots,N$, $k\in\NN$.} 
	\end{equation}
	Indeed,  by definition of prolongation and noting $\delta_h^2\equiv 1$, we have
	\begin{equation*}
	e_k(K_h x)=\delta_h^2 e_k(K_h x)=\delta_h \mathcal P_\delta e_k(x)=\delta_h 
	e_k(x).
	\end{equation*}
	for every $x\in \Omega$, $h=1,\dots,N$ and $k\in\NN$. 
	\\ Also remark that, in view of \eqref{fold2}, Assumption \ref{A2} also 
	implies 
	\begin{equation}\label{fold}
	\mathcal F_\delta e_k= \frac{1}{N} e_k. 
	\end{equation}
\end{remark}
\begin{example}
	Let $\Omega=(0,\pi)^2$ and $\Omega'=(0,2\pi)^2$. Consider the 
	transformations of $\RR^2$
	$$ \begin{array}{ll}
	K_1:= \text{id}, &~K_2: (x_1,x_2)\mapsto (-x_1+2\pi,x_2),\\
	K_3:(x_1,x_2)\mapsto (x_1,-x_2+2\pi),&~K_4: (x_1,x_2)\mapsto 
	-(x_1,x_2)+(2\pi,2\pi)  
	\end{array}.$$
	Then $\{\Omega,\{K_h\}_{h=1}^4\}$ is a tiling for $\Omega'$.  In 
	particular, Assumption \ref{A1} is satisfied: indeed setting 
	$\delta=(1,-1,-1,1)$ we have for each $\varphi\in H_0^1(\Omega')$
	$$\mathcal F_\delta \varphi(x)= 0 \quad \forall x\in \partial \Omega.$$
	Also note that the functions
	$$e_k(x):=\sin (k_1 x_1)\sin (k_2x_2)\quad k=(k_1,k_2)\in\NN^2$$
	satisfy Assumption \ref{A2}, indeed they are a base for $L^2(\Omega)$ 
	composed by eigenfunctions of $-\Delta$ in $H_0^1(\Omega)$ and
	$$e_k( K_h (x)):=\delta_h e_k(x)$$
	for all $x\in \RR^2$, $h=1,\dots,4$ and $k\in \NN^2$. The space 
	$\mathcal P_\delta L^2(\Omega)$ in this case coincides with the space of 
	so-called \emph{$(2,2)$-cyclic functions}, i.e.,  functions in $L^2(\Omega')$ 
	which are odd with respect to both axes $x_1=\pi$ and $x_2=\pi$.
	We refer to \cite{KomLor159} for some results on observability of wave 
	equation 
	with $(p,q)$-cyclic initial data. 
\end{example}

\vskip0.5cm
Our starting point is to show that, under Assumption \ref{A1} and Assumption \ref{A2}, 
the base of eigenfunctions $\{e_k\}$ is also a base of eigenfunctions also for 
an appropriate subspace of $L^2(\Omega')$, and to compute the associated 
coefficients.  
\begin{lemma}\label{lcr}
	Let $\Omega,\Omega'$ and $\{e_k\}$ satisfy Assumption \ref{A1} and 
	Assumption \ref{A2}.
	
	%
	
	Then $\{e_k\}\subset H^1_0(\Omega')$ and it is also a complete base for 
	$\mathcal P_\delta L^2(\Omega)$ formed by eigenfunctions of $-\Delta$ in 
	$\mathcal P_\delta H_0^1(\Omega')$.
	
	In particular, for every $k\in \NN$, if $u_k$ is the coefficient of 
	$u\in L^2(\Omega)$ (with respect to $e_k$) then $Nu_k$ is the coefficient of 
	$\mathcal P_\delta u$. 
\end{lemma}
\begin{proof} The proof is organized two steps. \\
	\emph{Claim 1: $\{e_k\}$ is a set of eigenfunctions of $-\Delta$ in 
		$H_0^1(\Omega')$}. Extending a result given in \cite{triangle}, we need to 
	show that, under Assumption \ref{A1} and Assumption \ref{A2}, if $e_k\in 
	H_0^1(\Omega)$ is a solution of the boundary value problem 
	$$ \int_{\Omega} \nabla e_k\nabla \varphi dx =\int_{\Omega} \gamma_k e_k 
	\varphi dx \quad \forall \varphi \in H_0^1(\Omega)
	$$
	for some $\gamma_k\in\RR$, then $e_k$ is also solution of the boundary value 
	problem on $\Omega'$
	$$ \int_{\Omega'} \nabla  e_k \nabla \varphi dx =\int_{\Omega'} \gamma_k e_k 
	\varphi dx \quad \forall \varphi \in H_0^1(\Omega').
	$$
	
	Now, recall from  Assumption \ref{A1} that if $\varphi \in H_1^0(\Omega')$ 
	then 
	$\mathcal F_\delta \varphi\in H_1^0(\Omega)$.  Then it follows again from 
	Assumption \ref{A1} and from Assumption \ref{A2} (in particular by recalling 
	that $K_h$'s are isometries and \eqref{opdef}) that for all $\varphi \in 
	H_0^1(\Omega')$
	\begin{align*}
	\int_{\Omega'} \nabla  e_k(x)&\nabla \varphi(x) dx= \int_{\bigcup_{h=1}^N K_h(
		\Omega)} 
	\nabla  e_k(x) \nabla \varphi(x) dx\\
	&=\sum_{h=1}^N \int_{\Omega} \nabla  e_k(K_h x) \nabla \varphi(K_h x) dx = 
	\int_{\Omega}  \nabla  e_k(x) \sum_{h=1}^N \delta_h \nabla \varphi(K_h x) dx\\
	&= \int_{\Omega}\nabla e_k(x) \nabla \mathcal F_\delta \varphi(x) dx= 
	\int_{\Omega} \gamma_k  e_k(x) \mathcal F_\delta \varphi(x) dx\\
	&= \int_{\Omega'} \gamma_k  e_k(x) \varphi(x) dx.
	\end{align*}
	and this completes the proof of Claim 1. 
	
	\emph{Claim 2: completeness of $\{e_k\}$ and computation of coefficients} By 
	Assumption \ref{A1} and Assumption \ref{A2} and by recalling $\delta_h^2=1$ for 
	each $h=1,\dots,N$, we have
	\begin{align*}
	\int_{\Omega'} \mathcal P_\delta u(x) e_k(x)dx&=\int_{\Omega'} \mathcal 
	P_\delta u(x) \mathcal P_\delta e_k(x)dx\\
	&=
	\sum_{h=1}^N\int_{K_h (\Omega)} \mathcal P_\delta u(x) \mathcal Pe_k(x)dx\\
	&=\sum_{h=1}^N\int_{K_h (\Omega)} \delta_h^2 u(K_h (x)) e_k(K_h (x))dx\\
	&=\sum_{h=1}^N\int_{\Omega}  u(x) e_k(x)dx=N\int_{\Omega}  u(x) e_k(x)dx,
	\end{align*}
	where the second to last equality holds because $K_h$'s are isometries. Then we may deduce two facts: first if $\{u_k\}$ are the 
	coefficients of $u\in L^2(\Omega)$ then $\{N u_k\}$ are coefficients of 
	$\mathcal P_\delta u$. Secondly, $\{e_k\}$ is a complete base for $\mathcal 
	P_\delta L^2(\Omega)$, indeed if the coefficients of $\mathcal P_\delta u$ 
	are 
	identically null, then also the coefficients of $u$ are identically null: since 
	$\{e_k\}$ is complete for $\Omega$ then $u\equiv 0$ and, consequently, 
	$\mathcal P_\delta u\equiv 0$, as well. 
\end{proof}
Next result establishes a relation between solutions of wave equations on tiles 
and their prolongations.
\begin{lemma}\label{lprol} 
	Let $\Omega,\Omega'$ and $\{e_k\}$ satisfy Assumption \ref{A1} and 
	Assumption \ref{A2}. 
	Let $u$ be the solution of
	\begin{equation}\label{wavegeneral}
	\begin{cases}
	u_{tt}-\Delta u=0& \text{in }\RR\times \Omega\\
	u=0&\text{in }\RR\times \partial \Omega\\
	u(t,0)=u_0,~u_t(t,0)=u_1&\text{in }\Omega
	\end{cases}
	\end{equation}
	Then $u$ is well defined in $\Omega\cup\Omega'$ and $\overline u=N 
	u|_{\Omega'}$ is the solution of 
	\begin{equation}\label{wavegeneral2}
	\begin{cases}
	\overline u_{tt}-\Delta \overline u=0& \text{in }\RR\times \Omega'\\
	\overline u=0&\text{in }\RR\times \partial \Omega'\\
	\overline u(t,0)=\mathcal P_\delta u_0,~\overline u(t,0)=\mathcal P_\delta 
	u_1&\text{in }\Omega'
	\end{cases}
	\end{equation}
	Conversely, if $\bar u$ is the solution of \eqref{wavegeneral2} then $\mathcal 
	F_\delta \bar u$ is the solution of  \eqref{wavegeneral} and for every 
	$h=1,\dots,N$ 
	\begin{equation}\label{sol}
	\mathcal F_\delta \bar u(t,x)= \frac{\delta_h}{N} \bar u(t,K_h x)\quad 
	\text{for each $x\in \Omega$}. 
	\end{equation}
	
\end{lemma}

\begin{proof}
	Let $\{\gamma_k\}$ be the sequence of eigenvalues associated to $\{e_k\}$ and 
	set $\omega_k=\sqrt{\gamma_k}$, for every $k\in\NN$. Expanding $u(t,x)$ with 
	respect to $e_k$ we obtain 
	$$u(t,x)=\sum_{k=1}^\infty (a_k e^{i\omega_k t}+b_k e^{-i\omega_k t})e_k(x)$$
	with $a_k$ and $b_k$  depending only the coefficients $c_k$ and $d_k$ of 
	$u_0$ and $u_1$ with respect to $\{e_k\}$.  In particular $a_k+b_k=c_k$ and 
	$a_k-b_k=-id_k/\omega_k$.
	We then have that the natural domain of $u$ coincides with the one of 
	$\{e_k\}$'s, hence it is included in $\Omega\cup\Omega'$. 
	By Lemma \ref{lcr}  the coefficients of $\mathcal P_\delta u_0$ and $\mathcal 
	P_\delta u_1$ are $Nc_k$ and
	$N d_k$, respectively. Then it is immediate to verify that
	$$N u(t,x)=\sum_{k=1}^\infty (N a_k e^{i\omega_k t}+ N b_k e^{-i\omega_k 
		t})e_k(x)$$
	is the solution of \eqref{wavegeneral2}. 
	
	Now, let 
	$$\bar u(t,x)= \sum_{k=1}^\infty (\bar a_k e^{i\omega_k t}+ \bar b_k 
	e^{-i\omega_k t}) e_k(x)$$
	be the solution of \eqref{wavegeneral2}, and note that, by the reasoning above, 
	setting $a_k:=\frac{1}{N}\bar a_k$ and $b_k:=\frac{1}{N}\bar b_k$ we have that 
	$$u(t,x):=\sum_{k=1}^\infty ( a_k e^{i\omega_k t}+ b_k e^{-i\omega_k t}) 
	e_k(x)=\frac{1}{N}\bar u(t,x)$$
	is the solution of \eqref{wavegeneral}. Hence to prove that $u(t,x)=\mathcal 
	F_\delta \bar u(t,x)$ it 
	it suffices to note that by Assumption \ref{A1} (see in particular 
	\eqref{fold}) 
	\begin{align*}
	\mathcal F_\delta \bar u(t,x)&=\sum_{k=1}^\infty (\bar a_k e^{i\omega_k t}+ 
	\bar b_k e^{-i\omega_k t}) \mathcal F_\delta 
	e_k(x)\\
	&=\frac{1}{N}\sum_{k=1}^\infty (\bar a_k e^{i\omega_k t}+ \bar b_k 
	e^{-i\omega_k t}) e_k(x)= \frac{1}{N}\bar u(t,x).
	\end{align*}
	Finally, we show \eqref{sol}: for each $h=1,\dots,N$ we have
	\begin{align*}
	\bar u(t,x)&=\delta_h^2 \bar u(t,x) =\delta_h \sum_{k=1}^\infty (\bar a_k 
	e^{i\omega_k t}+ \bar b_k e^{-i\omega_k t})  \delta _h e_k(x)\\
	&=\sum_{k=1}^\infty (\bar a_k e^{i\omega_k t}+ \bar b_k e^{-i\omega_k t}) 
	e_k(K_h x)=\delta_h \bar u(t,K_h x)
	\end{align*}
	and this concludes the proof.
	\eqref{wavegeneral2}.
\end{proof}
We are now in position to state the main result of this section, that bridges 
observability of tiles
with their prolongations.
\begin{theorem}\label{thmprol}
	Let $\Omega,\Omega'$ and $\{e_k\}$ satisfy Assumption \ref{A1} and 
	Assumption \ref{A2}. 
	Let $u$ be the solution of
	\begin{equation}\label{wavegeneralprol}
	\begin{cases}
	u_{tt}-\Delta u=0& \text{in }\RR\times \Omega\\
	u=0&\text{in }\RR\times \partial \Omega\\
	u(t,0)=u_0,~u_t(t,0)=u_1&\text{in }\Omega
	\end{cases}
	\end{equation}
	with $u_0,u_1\in L^2(\Omega)\times H^{-1}(\Omega)$ and let $\overline u$ be 
	the solution of
	\begin{equation}\label{wavegeneralprol2}
	\begin{cases}
	u_{tt}-\Delta u=0& \text{in }\RR\times \Omega'\\
	u=0&\text{in }\RR\times \partial \Omega'\\
	u(t,0)=\mathcal P_\delta u_0,~u_t(t,0)=\mathcal P_\delta u_1&\text{in 
	}\Omega'.
	\end{cases}
	\end{equation}
	Also let $\Omega'_0\subset \Omega'$ and define 
	$$\Omega_0:=\bigcup_{h=1}^N K_h^{-1}(\Omega'_0) \cap \Omega.$$
	Then for every $T>0$ and for every couple $c=(c_1,c_2)$ of positive constants, 
	the inequalities
	\begin{equation}\label{c1}
	\|u_0\|_{L^2(\Omega)}^2+\|u_1\|^2_{H^{-1}(\Omega)}\asymp_c \int_0^T\int_{\Omega_0} |u(t,x)|^2 dxdt. 
	\end{equation}
	hold if and only if 
	\begin{equation}\label{c2}
	\|\mathcal P_\delta u_0\|_{L^2(\Omega')}^2+\|\mathcal P_\delta u_1\|^2_{H^{-1}(\Omega')}\asymp_c \int_0^T\int_{\Omega'_0} |u(t,x)|^2 dxdt. 
	\end{equation}
\end{theorem}
\begin{proof}
	By Lemma \ref{lprol}, $u$ and $\overline u$ satisfy
	$$u(t,x)=\frac{\delta_h}{N}\overline u(t,K_h x)\quad \text{ for all 
	}h=1,\dots,N.$$
	Since $\Omega$ tiles $\Omega'$, then setting $\Omega_h:= K_h^{-1}(\Omega')\cap 
	\Omega$ we have $\Omega_0=\bigcup_{h=1}^N \Omega_h$ and $\Omega_0'=\bigcup_{h=1}^N K_h(\Omega_h)$, and that these unions are disjoint. Hence, also recalling 
	$|\delta_h|\equiv 1$ and that $K_h$'s are isometries, we have
	\begin{align*}
	\int_I\int_{\Omega_0'} |\overline u(t,x)|^2 dx&= \sum_{h=1}^N\int_I\int_{K_h 
		(\Omega_h)} |\overline u(t,x)|^2 dx\\
	&= \sum_{h=1}^N\int_I\int_{\Omega_h} |\overline u(t,K_h (x))|^2 dx\\
	&= N^2\sum_{h=1}^N\int_I\int_{\Omega_h} \left|\frac{\delta_h}{N}\overline u(t,K_h(x))\right|^2 
	dx\\
	&= N^2\sum_{h=1}^N\int_I\int_{\Omega_h} |u(t, x)|^2 dx\\
	&= N^2\int_I\int_{\Omega_0} |u(t, x)|^2 dx\\
	\end{align*}
	Finally, by Lemma \ref{lcr}  
	$$\|\mathcal P_\delta u_0\|^2_{L^2(\Omega')}= N^2\| u_0\|_{L^2(\Omega)}^2\quad\text{and}\quad 
	\|\mathcal P_\delta u_1\|^2_{H^{-1}(\Omega')}=N^2\| u_1\|_{H^{-1}(\Omega)}^2$$
	and this implies the equivalence between \eqref{c1} and \eqref{c2}.
\end{proof}

\section{Proof of Theorem \ref{thmmainintro}}\label{st1}
The proof of Theorem \ref{thmmainintro} is based on the application of Theorem 
\ref{thmprol} to the particular case 
$$\Omega=\mathcal T \quad \text{and}\quad \Omega'=\mathcal R.$$
%

We then need to admissibly tile $\mathcal R$ with $\mathcal T$ and a base 
$\{e_k\}$ formed by the eigenfunctions of $-\Delta$ in $H_0^1(\mathcal T)$ 
satisfying Assumption \ref{A2}. Such ingredients are provided in 
\cite{triangle}: in order to introduce them we need some notations. We consider 
the Pauli matrix
$$\sigma_z:=\begin{pmatrix}
1&0\\
0&-1\\
\end{pmatrix}$$
and the rotation matrix 
$$R_\alpha:=\begin{pmatrix}
\cos \alpha &\sin\alpha\\
-\sin\alpha&\cos\alpha\\
\end{pmatrix}$$
where $\alpha:=\pi/3$. Now let $v_1:=(0,1/\sqrt{3})$ and $v_2:=(0,1)$ be two of 
the three vertices of $\mathcal T$ and define the transformations from 
$\RR^2$ onto itself
\begin{equation}\label{kdef}
{\begin{array}{ll}
	K_1:=id;\quad &\quad K_4:x\mapsto- R_\alpha(x-v_2)+ 
	3v_1\\
	K_2: x\mapsto -R_\alpha \sigma_z(x-v_2)+v_2;&\quad 
	K_5:  x\mapsto -R_\alpha(x-v_2)+3 v_1+v_2\\
	K_3: x\mapsto R_\alpha(x-v_2)+v_2;&\quad K_6: x\mapsto 
	-x+3 v_1+v_2
	\end{array}}
\end{equation}
and note $(\mathcal T, \{K_h\}_{h=1}^6)$ is a tiling for $\mathcal R$. Indeed 
\begin{equation}\label{a1triangle}
cl( \mathcal R)=\bigcup_{h=1}^6 K_h cl(\mathcal T),
\end{equation}
and the sets $K_h \mathcal T$, for $h=1,\dots,6$, do not overlap -- see Figure 
\ref{folding} and \cite{triangle}. 

\begin{figure}
	\centering\includegraphics[scale=0.6]{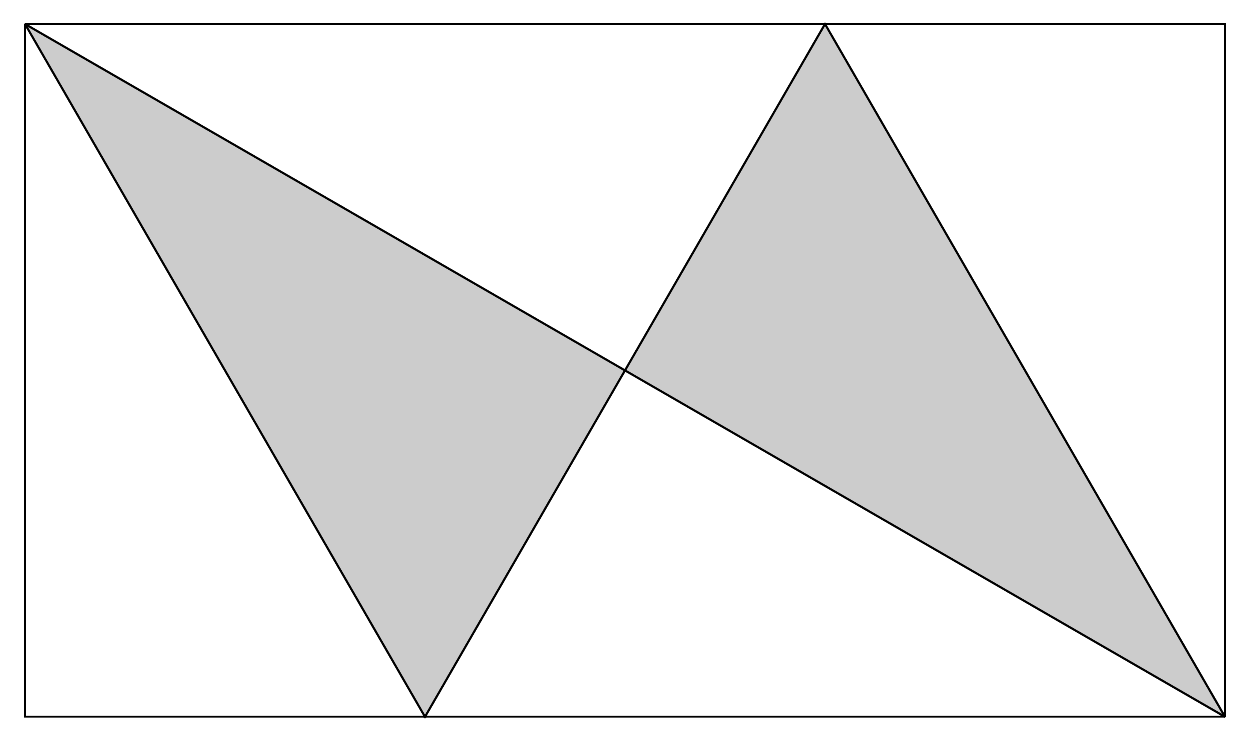}
	\caption{The tiling of $\mathcal R$ with $\mathcal T$, the grey areas 
		correspond to negative $\delta_h$'s.  \label{folding}}\end{figure}
We set 
$$\delta:=(1,-1,1,1,-1,1).$$ 
and, in next result, we prove that $\mathcal T$ admissibly tiles $\mathcal R$. 

\begin{lemma}\label{ladm}
	$(\mathcal T,\{K_h\}_{h=1}^6)$ is an admissible tiling of $\mathcal R$.
\end{lemma}
\begin{proof}
	We want to show that if $\varphi\in H_0^1(\mathcal R)$ then $\mathcal F_\delta 
	\varphi\in H_0^1(\mathcal T)$. To this end let $v_0:=(0,0)$, 
	$v_1:=(1/\sqrt{3},0)$ and $v_2:=(0,1)$ be the vertices of $\mathcal T$ and 
	define 
	$$x_{ij}^\lambda:=\lambda v_i +(1-\lambda) v_{j}.$$
	so that $\partial \mathcal T= \{x_{ij}^\lambda\mid
	\lambda\in[0,1], 0\leq i<j\leq2\}$. 
	By a direct computation, for all $\lambda\in[0,1]$ 
	$$K_1(x_{01}^\lambda), K_6(x_{01}^\lambda)\in \partial \mathcal R,$$
	$$K_2(x_{01}^\lambda)=K_4(x_{01}^\lambda),$$
	and
	$$K_3(x_{02}^\lambda)=K_5(x_{02}^\lambda).$$
	Since $\varphi\in H_0^1(\mathcal R)$ then $\mathcal F_\delta 
	\varphi(x_{01}^\lambda)=0$.
	Similarly, for all $\lambda\in[0,1]$
	$$K_1(x_{02}^\lambda), K_6(x_{02}^\lambda)\in \partial \mathcal R,$$
	$$K_2(x_{02}^\lambda)=K_3(x_{02}^\lambda),$$
	and
	$$K_4(x_{02}^\lambda)=K_5(x_{02}^\lambda)$$
	therefore $\mathcal F_\delta \varphi(x_{02}^\lambda)=0$ for all 
	$\lambda\in[0,1]$.
	Finally for all $\lambda\in[0,1]$ 
	$$K_3(x_{12}^\lambda), K_4(x_{12}^\lambda)\in \partial \mathcal R,$$
	$$K_1(x_{12}^\lambda)=K_2(x_{12}^\lambda),$$
	and
	$$K_5(x_{12}^\lambda)=K_6(x_{12}^\lambda)$$
	therefore we get also in this case $\mathcal F_\delta 
	\varphi(x_{12}^\lambda)=0$ for all $\lambda\in[0,1]$ and we may conclude that 
	$\mathcal F_\delta \varphi\in H_0^1(\mathcal T)$. 
\end{proof}
\begin{remark}
	Lemma \ref{ladm} was remarked in \cite[p.312]{triangle}, but to the 
	best of our knowledge, this is the first time an explicit proof is provided. 
\end{remark}

Now, consider the eigenfunctions of $-\Delta$ in $H_0^1(\mathcal R)$:
$$\overline e_k(x_1,x_2):=\sin(\pi k_1 \frac{x_1}{\sqrt{3}})\sin(\pi k_2 x_2), 
\quad k=(k_1,k_2)\in\NN^2.$$
We finally define for every $k\in \NN^2$
\begin{equation}\label{def}
e_k(x):=N^2 \mathcal F_\delta \overline e_k= \sum_{h=1}^6 \delta_h \overline 
e_k(K_h x). 
\end{equation}

Next result, proved in \cite{triangle}, states that Assumption \ref{A2} is 
satisfied by $\{e_k\}$.
\begin{lemma}
	The set of functions $\{e_k\}$ defined in \eqref{def} is a complete 
	orthogonal 
	base for $\mathcal T$ 
	formed by the eigenfunction of $-\Delta$ in $H_0^1(\mathcal T)$.   
	Furthermore $\mathcal P_\delta e_k(x)=e_k(x)$.
\end{lemma}

\begin{remark}
	For each $k\in\NN^2$, the eigenfunctions $e_k$ and $\bar e_k$ share the same 
	eigenvalue $\gamma_k=\pi^2(\frac{k_1^2}{3}+k_2^2)$, see \cite{triangle}.
\end{remark}
Next gives access to classical results on observability of rectangular 
membranes 
for the study of triangular domains. 
\begin{theorem}\label{thmmain}
	Let $u$ be the solution of \eqref{wave} 
	with $u_0,u_1\in L^2(\Omega)\times H^{-1}(\Omega)$ and let $\overline u$ be 
	the solution of
	\begin{equation*}
	\begin{cases}
	u_{tt}-\Delta u=0& \text{on }\RR\times \R\\
	u=0&\text{in }\RR\times \partial \R\\
	u(t,0)=\mathcal P_\delta u_0,~u_t(t,0)=\mathcal P_\delta u_1&\text{in 
	}\R.
	\end{cases}
	\end{equation*}
	Also let $\R_0\subset \R$ and define 
	$$\T_0:=\bigcup_{h=1}^N K_h^{-1}(\T_0) \cap \Omega.$$
	Then for every $T>0$ and for every couple $c=(c_1,c_2)$ of positive constants, 
	the inequalities
	\begin{equation}\label{c1r}
	\|u_0\|_{L^2(\T)}^2+\|u_1\|^2_{H^{-1}(\T)}\asymp_c \int_0^T\int_{\T_0} |u(t,x)|^2 dxdt. 
	\end{equation}
	hold if and only if 
	\begin{equation}\label{c2r}
	\|\mathcal P_\delta u_0\|_{L^2(\R)}^2+\|\mathcal P_\delta u_1\|^2_{H^{-1}(\R)}\asymp_c \int_0^T\int_{\R_0} |u(t,x)|^2 dxdt. 
	\end{equation}
\end{theorem}
\begin{proof}
	Since $\mathcal T,\mathcal R$ and $\{e_k\}$ satisfy Assumption \ref{A1} and 
	Assumption \ref{A2}, then the claim follows by a direct application of Theorem 
	\ref{thmprol} with $\Omega=\mathcal T$ and $\Omega'=\mathcal R$.
\end{proof}

We conclude this section by showing that Theorem \ref{thmmainintro} is a direct 
consequence of Theorem \ref{thmmain}:
\begin{proof}[Proof of Theorem \ref{thmmainintro}] 
	By Lemma \ref{lcr}, if $(u_0,u_1)\in L^2(\mathcal T)\times H^{-1}(\mathcal T)$ 
	then 
	$(\mathcal P_\delta u_0,\mathcal P_\delta u_1)\in L^2(\mathcal R)\times 
	H^{-1}(\mathcal R)$. The claim hence follows by Theorem \ref{thmmain}. 
\end{proof}

\section{Acknowledgements}
The author is grateful to Vilmos Komornik (Universit\'e de Strasbourg) and Paola Loreti (Sapienza Universit\`{a} di Roma) for the fruitful discussions on the topic of internal observability of wave equations and related symmetry properties of the initial data: they provided insight and expertise that greatly assisted the research.

\end{document}